\documentclass{amsart}
\usepackage{amsfonts,amssymb,amsmath,amsthm}
\usepackage{url}
\usepackage{enumerate}
\usepackage{natbib}

\newtheorem{thrm}{Theorem}[section]
\newtheorem{lem}[thrm]{Lemma}

\theoremstyle{definition}
\newtheorem{definition}[thrm]{Definition}
\newtheorem{remark}[thrm]{Remark}
\numberwithin{equation}{section}

\author{E. Shemyakova}
\address{
Mathematics Department, University of Western Ontario, London, Ontario, Canada
\&
Dorodnicyn Computing Centre of Russian Academy of Sciences, Moscow, Russia.
}
\email{shemyakova.katya@gmail.com}
\thanks{}

\keywords{Completeness of Darboux Wronskian formulas, completeness of Darboux determinants, 
Darboux transformations, invariants for solution of PDEs}
\subjclass{Primary  53Z05, Secondary 35Q99}
\newcommand{\Sym}{\ensuremath \mathrm{Sym}}

\newcommand{\Ker}{\ensuremath \mathop \mathrm{Ker} \nolimits}

\def\o#1{\ensuremath \mathcal{#1}}
\usepackage{latexsym,amsmath,amssymb}

\begin{document}

\title[Completeness of Darboux Wronskian Formulae]{Proof of the Completeness of Darboux Wronskian Formulae for Order Two}

\begin{abstract}
Darboux Wronskian formulas allow to construct Darboux transformations, but Laplace transformations, which are Darboux transformations of order one
cannot be represented this way. 
It has been a long standing problem on what are other exceptions. In our previous work we proved that among transformations of total 
order one there are no other exceptions. Here we prove that for transformations of total order two there are no exceptions at all.
We also obtain a simple explicit invariant description of all possible Darboux Transformations of total order two. \\
Mathematics Subject Classification Numbers: 53Z05, 35Q99.
\end{abstract}
\maketitle

\section{Introduction}
\label{intro}
Classical Darboux transformations and their generalizations are methods for obtaining analytic solutions of linear Partial Differential Equations (PDEs).
They also serve as a leverage for larger theories for solution of non-linear PDEs, see for example~\cite{matveev:1991:darboux} and references therein.

In the present paper we are concerned with the intertwining relations $\o{N} \circ \o{L} = \o{L}_1 \circ \o{M}$ for operators of the form 
\begin{equation} \label{op:L}
\o{L}  = D_{x} D_y + a D_x + b D_y + c  \ ,
\end{equation}
where the coefficients may be non constant. Since they have been introduced in the classical work of~\cite{Darboux2},
we shall call them Darboux transformations too. PDEs corresponding to such operators appear also
as part of the problem of the search of flat metrics, see~\cite{krichever:97}. 

Given Linear Partial Differential Operator (LPDO) $\o{L}$ and some LPDO $\o{M}$, the coefficients of the resulting operator $\o{L}_1$ and 
of the auxiliary operator $\o{N}$ can be found algebraically. There are two choices of $\o{M}$, which always lead to a DT for a given operator~\eqref{op:L}:
$\o{M} = D_x + b$, and $\o{M} = D_y + a$. 
These Darboux transformations have a special name: Laplace transformations. The latter are not be confused with Laplace integral transforms.

There is also a large class of Darboux transformations generated by operators $\o{M}$ that are constructed using so-called Darboux Wronskian formulas. 
These are based on the assumption that we know some number of linearly independent particular solutions of the initial PDE, $\o{L} \psi =0$. 
This class is a very large class and Darboux transformations of arbitrary orders can be constructed provided we know enough number of particular solutions. 
Laplace transformations, which are Darboux transformations of order one do not belong to this class. 

Laplace transformations are particularly good Darboux transformations, see~\cite{ts:genLaplace05}, and they have been the only known
examples of Darboux transformations that cannot be described by Darboux Wronskian formulas. 
In~\cite{Laplace_only_degenerate_2012} we have proved that a Darboux transformation of total order one is either described by Darboux Wronskians or is a Laplace transformation.
The problem reduces to solution of a non-linear PDE. The PDE was not so large and noticing some interesting structure we were able
to tackle the problem.

After that work it was still unclear whether there are some exceptional transformations, that is such that cannot be described by 
Darboux Wronskian formulas among Darboux transformations of orders higher than one. 
This problem is reducing to solution of a system of two large non-linear PDEs, for which methods of the previous work~\cite{Laplace_only_degenerate_2012}
were hard to apply. We, however, succeeded in proving that Darboux Wronskian formulas complete for transformations of order two in a different and rather elegant fashion, and 
present this proof in this paper. 

%

Recently there have been several new ideas to tackle Darboux transformations and related problems. Thus, \cite{tsarev2008categories_forDT} and~\cite{Cas:Singer2011} have made very important progress in the description of factorizable operators 
corresponding to linear PDEs in terms of certain abelian categories and algebraic groups, respectively. In the present paper, we adopt
an approach that is based on ideas of Differential Geometry, and is constructive. 

Our main result is an elegant proof that all Darboux transformations of total order two can be described by Wronskian formulae
(Theorem~\ref{thm:completeness}). The second achievement is an easy to use invariant description of all these Darboux transformations 
(Theorem~\ref{thm:simple}). 

The paper is organized as follows. Darboux transformations of total order two are those that has $\o{M}$ 
in one of the following forms:
\begin{eqnarray*}
\o{M}&=& m_{20}D_{xx} + m_{10}D_{x}  +m_{00}\ ,\\
\o{M}&=& m_{02}D_{yy} + m_{01}D_{y} +m_{00}\ ,\\
\o{M}&=& m_{10}D_{x} +m_{01}D_{y} +m_{00}\ ,
\end{eqnarray*}
where the $m_{ij} \in K$ are not necessarily constant. 
In Sec.~\ref{sec:proj} we show that to cover all Darboux transformations of total order two it is enough to consider $\o{M}$ is the form $D_x+qD_y+r$, where $p$ and $q$ are some functions.
After some preparation in the next two sections, we introduce in Sec.~\ref{sec:gauge:evolution} new transformations of the pair $\{\o{L}, \o{M}\}$, which we name \emph{gauged evolution}.
We determine the generating invariants uniquely defining the equivalence classes under these transformations and use them to invariantize
the nonlinear system of PDEs defining all possible Darboux transformations of total order two. 
The invariantized system is easier and can be solved explicitly by classical methods, however, even though we have a technical solution, it is in quadratures and 
it is useful neither for invariant description of Darboux transformations, nor to judge whether Wronskian Formulae give all such Darboux transformations or not.

Therefore, we need a further invention, Theorem~\ref{thm:i30}, through which we are able to obtain an elegant general solution
(Theorem~\ref{thm:simple}) of the invariantized system of PDE.
We still have to remember that even if the invariantized system of PDEs has solutions, the existence of Darboux transformations depends also on
the existence of a solution of a nonlinear PDE system~(\ref{sys:inv:evolution}),
where we return from gauged evolution invariants to the coefficients
of operators $\o{L}$ and $\o{M}$.
In the proof of Theorem~\ref{thm:completeness} we resolve this problem and conclude that
for every Darboux transformation of total order 2 there exist two linearly independent partial
 solutions of $\o{L} u =0$, such that it can be constructed using Darboux Wronskian formulas.

\section{Preliminaries}

Let $K$ be a differential field of characteristic zero with commuting derivations $\partial_x, \partial_y$.
Let $K[D]=K[D_x, D_y]$ be the corresponding ring of linear
partial differential operators over $K$, where $D_x, D_y$ correspond to derivations $\partial_x, \partial_y$.

Operators $\o{L} \in K[D]$ have the general form
$\o{L} = \sum_{i+j =0}^d a_{ij}D_x^i D_y^j$, where $a_{ij} \in K$. The formal polynomial
$\Sym_\o{L} =  \sum_{i+j = d} a_{ij} X^i Y^j$ in some formal variables $X, Y$ is called the \emph{symbol} of $\o{L}$.

One can either assume field $K$ to be either differentially closed, in other words containing all the solutions
of, in general nonlinear, Partial Differential Equations (PDEs)
with coefficients in $K$, or simply assume that $K$ contains the solutions of those PDEs that we encounter on the way.

\begin{definition}
An operator $\o{L}_1 \in K[D]$ is called a {\it Darboux transformation}
of an operator $\o{L} \in K[D]$, if $\Sym(\o{L}) =\Sym (\o{L}_1)$, and there exist operators
$\o{N} \in K[D]$ and $\o{M} \in K[D]$ such that
\begin{equation} \label{eq:main}
 \o{N} \circ \o{L} = \o{L}_1 \circ \o{M} \ .
\end{equation}
In this case we say that this Darboux transformation corresponds to pair $\{\o{L}, \o{M}\}$, and
that operator $\o{L}_1$ is \emph{associated with}, or \emph{Darboux-conjugated to}
operator $\o{L}$
and use the notation
\[
 \o{L}_1 = \varphi(\o{L}, \o{M}, \o{N}) \ .
\]
Note that coefficients of the operators are not required to be constants.
\end{definition}
Darboux transformation implies the following transformations of kernels: $\Ker \o{L} \to \Ker \o{L}_1: \; \psi  \mapsto \o{M}(\psi)$
and requires $\Sym(\o{M}) =\Sym (\o{N})$.
\begin{definition} \label{def:lapl} The Darboux transformation of an operator~\eqref{op:L},
where $a,b,c$ are not required to be constants, is called a \emph{Laplace transformation}
if the corresponding operator $\o{M}$ is either $\o{M} = D_x + b$, or $\o{M} = D_y + a$.
\end{definition}
Laplace transformations are the most well-studied case of Darboux transformation and have several important properties, see~\cite{Darboux2}.

One of the most famous results in~\cite{Darboux2} concerns Darboux transformations for operators of the form~(\ref{op:L}) and can be formulated as follows.
\begin{thrm}[Darboux]
\label{thm:main:Darboux}
Let we $\o{L}$ be an operator of the form~\eqref{op:L} and $\psi_1, \dots , \psi_{m+n} \in \Ker \o{L}$ be linearly independent then
\begin{equation} \label{eq:wron}
\o{M} (\psi)=W_{m,n}(\psi, \psi_1, \dots , \psi_{m+n})
\end{equation}
defines some Darboux transformation for operator $\o{L}$.
\end{thrm}
\noindent
Here $W_{m,n}$ is a Wronskian-like function $W_{m,n}(\psi, \psi_1, \dots , \psi_{m+n}) = $
\[
\left|
\begin{array}{lllllll}
 \psi & D_x\psi &    \dots &  D^m_x\psi & D_y \psi  &   \dots &  D^n_y\psi \\
 \psi_1 & D_x\psi_1  &   \dots &  D^m_x\psi_1 & D_y \psi_1 &   \dots &  D^n_y\psi_1 \\
  \dots  &  \dots&  \dots&    \dots&  \dots&  \dots&   \dots \\
 \psi_{m+n} & D_x\psi_{m+n}  &   \dots &  D^m_x\psi_{m+n} & D_y \psi_{m+n}  &   \dots &  D^n_y\psi_{m+n} \\
\end{array}
\right|
\]
That is a Darboux transformation of order $m+n$ can be built using $m+n$ particular solutions of the initial equation $\o{L}(\psi)=0$, see~\cite{Darboux2}.
Darboux Wronskian-like formulas~(\ref{eq:wron}) provides a large class of possible Darboux transformations.

\section{Normalization of Darboux transformations}
\label{sec:proj}
\subsection{Normalization of Darboux transformations using expansion}
\begin{lem}
\label{lem:AL}
Let $\o{L}$ be of the form~\eqref{op:L} and $\o{M}$ be an operator of arbitrary order from $K[D]$.
Let $\o{M}$ define at least one Darboux transformation for operator $\o{L}$,
then for any given operator $\o{A} \in K[D]$ there exists also a Darboux transformation for the same $\o{L}$ with
$\o{M}=\o{M}+\o{A} \circ \o{L}$.
\end{lem}
\begin{proof} Equality~(\ref{eq:main}) implies that
\[\o{L}_1 \circ (\o{M} +\o{A} \circ \o{L}) =  \o{M}_1 \circ \o{L} + \o{L}_1 \circ \o{A} \circ \o{L} =
(\o{M}_1 + \o{L}_1 \circ \o{A}) \circ \o{L}\]
 is true for an arbitrary operator $\o{A} \in K[D]$.
 Therefore, there exists a Darboux transformation for $\o{L}$ with $\o{M}=\o{M}+\o{A} \circ \o{L}$.
\end{proof}
\begin{definition}
Lemma~\ref{lem:AL} describes transformations of pairs of operators $\{\o{L},\o{M}\}$. It shows that such transformations
preserve the property of the existence of Darboux transformations for a given operator $\o{L}$,
 and splits the operators $\o{M}$ into equivalence classes.
 We name this transformation an \textit{expansion}.
\end{definition}
\begin{remark} Notice that the resulting operators of the initial Darboux transformation and of the one generated for $\o{L}$ by $\o{M}+\o{A} \circ \o{L}$
are the same.
\end{remark}

Given operator $\o{L}$ of the form~\eqref{op:L}, we shall be considering different pairs  $\{\o{L},\o{M}\}$, where 
$\o{M}$ are operators in $K[D]$. Using expansion we can eliminate all the mixed derivatives in $\o{M}$ in the case of $\o{L}$ of the form (\ref{op:L}).
\begin{definition} Let $\o{L} \in K[D]$ be  of the form (\ref{op:L}) and $\o{M} \in K[D]$ be an arbitrary operator, then
denote the result of elimination of the mixed derivatives in $\o{M}$ using $\o{L}$
as $\pi_{\o{L}}(\o{M})$.
\end{definition}
\begin{definition} Given $\o{L} \in K[D]$ of the form (\ref{op:L}), we define the bi-degree $\deg_{\o{L}} \o{M} =(m,n)$
of operator $\o{M}$ with respect to $\o{L}$ as follows:
$m$ is the highest derivative with respect to $D_x$ in $\pi_{\o{L}}(\o{M})$  and $n$ is that with respect to $D_y$.
We shall say that $m+n$ is the \textit{total degree} of $\o{M}$.
\end{definition}
\begin{definition} By the \textit{degree} or \textit{total degree} of a Darboux transformation $\o{L}_1=\varphi(\o{L}, \o{M}, \o{N})$ we shall understand
the degree or the total degree of $\o{M}$.
\end{definition}
\subsection{Normalization of Darboux transformations using composition with Laplace transformations}
\begin{definition}
Let there be a Darboux transformation of arbitrary $\o{L} \in K[D]$ defined by some $\o{M} \in K[D]$, that is~(\ref{eq:main}) holds. Let
the result, operator $\o{L}_1 \in K[D]$ be transformed into some $\o{L}_2 \in K[D]$ by a Darboux transformation defined by some $\o{M}_1 \in K[D]$, that is
$\o{N}_1 \circ \o{L}_1 = \o{L}_2 \circ \o{M}_1$ for some $\o{N}_1 \in K[D]$. Then \textit{the composition of these two Darboux transformations}
is a Darboux transformation transforming $\o{L}$ into $\o{L}_2$ defined by $\o{N}_1 \circ \o{N} \circ \o{L} =\o{L}_2 \circ \o{M}_1\circ \o{M}$.
\end{definition}

The following lemma allows us to use expansion and composition together. 
\begin{lem}[Correctness of the composition of two Darboux transformation with expansion]
The result of composition of two Darboux transformations does not depend on the choice of the operator $\o{M}$ within its class of equivalence
under expansion.
\end{lem}
\begin{proof} Let there be a Darboux transformation of arbitrary $\o{L} \in K[D]$ defined by some $\o{M} \in K[D]$, i.e.,~(\ref{eq:main}) holds.
Let
the result, operator $\o{L}_1 \in K[D]$ be transformed into some $\o{L}_2 \in K[D]$ by a Darboux transformation defined by some $\o{M}_1 \in K[D]$,
i.e.,~$\o{N}_1 \circ \o{L}_1 = \o{L}_2 \circ \o{M}_1$ for some $\o{N}_1 \in K[D]$.
Consider $(\o{M} + \o{A} \circ \o{L})$ and $\o{M}_1 + \o{B} \circ \o{L}_1$ for some $\o{A}, \o{B} \in K[D]$,
which belongs to the same classes of equivalence under the expansion as $\o{M}$ and $\o{M}_1$ correspondingly.
That is we have:
\begin{eqnarray*}
 (\o{N} + \o{L}_1 \circ \o{A}) \circ \o{L} &=& \o{L}_1 \circ (\o{M} + \o{A} \circ \o{L})\ , \nonumber \\
 (\o{N}_1 + \o{L}_2 \circ \o{B}) \circ \o{L}_1 &=& \o{L}_2 \circ (\o{M}_1 + \o{B} \circ \o{L}_1)\ . \label{DT2}
\end{eqnarray*}
Then the composition is
\[(\o{N}_1 + \o{L}_2 \circ \o{B}) \circ  (\o{N} + \o{L}_1 \circ \o{A}) \circ \o{L} =
   (\o{N}_1 + \o{L}_2 \circ \o{B}) \circ \o{L}_1 \circ (\o{M} + \o{A} \circ \o{L}) \ ,
\]
which using equality~(\ref{DT2}) can be re-written as
\[
(\o{N}_1 + \o{L}_2 \circ \o{B}) \circ  (\o{N} + \o{L}_1 \circ \o{A}) \circ \o{L} =
  \o{L}_2 \circ (\o{M}_1 + \o{B} \circ \o{L}_1) \circ (\o{M} + \o{A} \circ \o{L}) \ .
\]
After expanding some multiples and
re-grouping we have
\[(\o{N}_1 \circ \o{N} + \o{L}_2 \circ \o{C} + \o{N}_1 \circ  \o{L}_1 \circ \o{A})
  \circ \o{L} = \o{L}_2 \circ (\o{M}_1 \circ \o{M} + \o{E} \circ \o{L} + \o{B} \circ \o{L}_1 \circ \o{M}) \ ,
\]
where $\o{C}=\o{B} \circ \o{N}+\o{B} \circ \o{L}_1 \circ \o{A}$, and $\o{E}=\o{M}_1 \circ \o{A}+\o{B} \circ \o{L}_1 \circ  \o{A}$.
Finally, substituting $\o{N} \circ \o{L}$ instead of $\o{L}_1 \circ \o{M}$ we obtain that the ``$\o{M}$'' operator of this Darboux transformation belongs to the same equivalence class that  $\o{M}_1 \circ \o{M}$ does under the expansion transformation.
\end{proof}

Then one of the results of~\cite{Darboux2} can be interpreted as follows:
\begin{thrm}
Let $\o{L} \in K[D]$ be of the form (\ref{op:L}) and $\o{M} \in K[D]$
define a Darboux transformation for $\o{L}$, and $\deg_{\o{L}} \o{M} =(m,n)$. Let
$\o{M}_x = D_y + a$ and $\o{M}_y =D_x + b$ define LTs for operator~(\ref{op:L}). Then
\begin{enumerate}
\item $\deg \pi(\o{M} \circ \o{M}_x) = (m-1,n+1)$,
\item $\deg \pi( \o{M} \circ \o{M}_x) = (m+1,n-1)$,
\item $\pi(\o{M}_x \circ \o{M}_y) = b_y -c+ab$, which is an operator of order zero,
\item $\pi(\o{M}_y \circ \o{M}_x) = a_x -c+ab$, which is an operator of order zero.
\end{enumerate}
\end{thrm}
Summarizing all the results we can formulate the following theorem.
\begin{thrm}
\label{thm:proj}
Let $\o{L} \in K[D]$ be of the form (\ref{op:L}) and $\o{M} \in K[D]$ define a Darboux transformation
for $\o{L}$. Let $\deg_{\o{L}} \o{M} =(m,n)$. Then for every $i = 1, \dots,\min(m,n)$, there exists an operator
$\o{M}_i$ without mixed derivatives having the property that $\deg_{\o{L}} \o{M}_i =(m-i,n+i)$ and $\o{M}_i$ defines a Darboux transformation of $\o{L}$.
\end{thrm}

\begin{lem}[$\o{M}$ can be multiplied by a function on the left]
\label{lem:div}
Let there exist a Darboux transformation of operator $\o{L} \in K[D]$ with some operator $\o{M} \in K[D]$.
Then for every invertible element $p \in K$ there exists a Darboux transformation of operator $\o{L}$ with operator $p\o{M}$.
\end{lem}
\begin{proof} The conditions of the lemma imply that, for some $\o{N},\o{L}_1 \in K[D]$,
equality~(\ref{eq:main}) holds.
Therefore,
$p \circ \o{N} \circ \o{L} = p \circ \o{L}_1 \circ p^{-1} \circ p \circ \o{M}$ is true also.
Since the symbol of $\o{L}_1$ is not altered under gauge transformations,
 then operator $p \circ \o{L}_1 \circ p^{-1}$ is
an operator of the form~(\ref{op:L}) and we have proved the statement of the lemma.
\end{proof}

Let $\o{L} \in K[D]$ be of the form (\ref{op:L}) and $\o{M} \in K[D]$ of arbitrary form and order defining a Darboux transformation
for $\o{L}$. Theorem~\ref{thm:proj} and Lemma~\ref{lem:div} imply that
using operations of expansion, composition with LTs, and division by a function on the left, we can bring such Darboux transformation
into a normalized form with $\o{M}$ having no mixed derivatives and having one of the following symbols:
\[
\begin{array}{ll}
 \Sym(\o{M}) = &X^k, \quad k>0 \ , \\
 \Sym(\o{M}) = &Y^k, \quad k>0 \ ,\\
 \Sym(\o{M}) = &X^k + qY^k, \quad k>0, \; q \ne 0 \ .\\
\end{array}
\]
Before we decide which of these to use in further considerations, let us consider the uniqueness problem for Darboux transformations.

\subsection{Uniqueness of Darboux transformations for given $\o{L}$ and $\o{M}$}

\begin{thrm} \label{thm:2} Let $\o{L} \in K[D]$ be of the form (\ref{op:L}) and let $\o{M} \in K[D]$
define some Darboux transformation. Then, unless for its normalized form we have $\Sym(\o{M}) = X^k$ or $\Sym(\o{M})=Y^k$,
such a Darboux transformation is unique.

If for its normalized form $\Sym(\o{M}) = X^k$ (corresp. $\Sym(\o{M})=Y^k$) and there exist two Darboux transformations:
$\o{L}_1=\varphi(\o{L}, \o{M}, \o{N})$ and
$\o{L}_1+\o{L}'_1=\varphi(\o{L}, \o{M} + \o{M}', \o{N}+ \o{N}')$, then
\begin{equation} \label{conds}
\left.
\begin{array}{ll}
                          & \Sym(M_1') = X^{k-1},  L_1' = D_y + \gamma \ , \\
\left( \text{corresp.}  \quad \; \right. & \left. \Sym(M_1') =  Y^{k-1},  L_1' = D_x + \gamma \; \right)  \ .\\
\end{array}
\right\}
\end{equation}
for some $\gamma \in K$.
\end{thrm}
\begin{proof} Since $\o{L}$ is of the form~(\ref{op:L}), for $\o{L}'_1$ there are only four possibilities:
\[
\begin{array}{ll}
 \o{L}_1' = & D_x + \beta D_y + \gamma, \quad \beta \ne 0 \ , \\
 \o{L}_1' = & D_x + \gamma \ , \\
 \o{L}_1' = & D_y + \gamma \ , \\
 \o{L}_1' = & 1 \ . \\
\end{array}
\]
$\o{L}_1=\varphi(\o{L}, \o{M}, \o{N})$ and $\o{L}_1+\o{L}'_1=\varphi(\o{L}, \o{M} + \o{M}', \o{N}+ \o{N}')$ implies
\begin{equation} \label{eq:cut}
 \o{M}_1' \circ \o{L} = \o{L}_1' \circ \o{M} \ .
\end{equation}
Case $\o{L}_1'=1$ cannot take place because if it does then $\o{M}_1' \circ \o{L} = \o{M}$, which is impossible as
$\Sym(M)$ cannot be divisible by $XY$.

Let $\Sym(\o{M}) = X^k$, $k>0$, then
\[
 \Sym(\o{M}_1') \cdot X \cdot Y = \Sym(\o{L}_1') \cdot X^k\ ,
\]
which implies that $\Sym(\o{L}'_1)$ must be divisible by $Y$, which is only possible if
$\o{L}_1' = D_y + \gamma$. Then $\Sym(\o{M}_1')$ cannot contain extra $Y$-s, and therefore,
$\Sym(\o{M}_1') = X^{k-1}$.

Analogously, if $\Sym(\o{M}) = Y^k$, $k>0$, we have $\o{L}_1' = D_x + \gamma$ and $\Sym(\o{M}_1') = Y^{k-1}$.

Let $\Sym(M) = X^k + qY^k$, then
\[
 \Sym(\o{M}_1') \cdot X \cdot Y = \Sym(\o{L}_1') \cdot (X^k + qY^k)\ ,
\]
which means that $\Sym(\o{L}_1')$ must be divisible by $XY$, which is impossible.
\end{proof}

Theorem~\ref{thm:2} guarantees uniqueness of a Darboux transformation for given $\o{M}$ and $\o{L}$ if
$\Sym(\o{M}) = X^k + qY^k$. Further below we shall be interested in Darboux transformation of the total degree
two, and we choose the normal form for such transformations with
\[
\Sym(\o{M}) = X + q Y \ .
\]

\section{Existence of a Darboux transformation defined by $\o{M}$ of bi-degree $(1,1)$}

Even for the simplest case of $\o{M}$ of total degree $1$, the problem of describing all Darboux transformations is not easy~\cite{Laplace_only_degenerate_2012}.
For the case of $\o{M}$ of total degree $2$, which is considered here the problem becomes very difficult.
\begin{thrm}
Let $\o{L}$ be of the form (\ref{op:L}) and $\o{M} \in K[D]$ in the form
\begin{equation} \label{op:M}
 \o{M} =D_x + qD_y + r \ .
\end{equation}
If there exists a corresponding Darboux transformation, then the corresponding operator $\o{N}$ is given by $\o{N} = \o{M} - (\ln q)_x + q_y$.
The necessary and sufficient conditions for the existence of a Darboux transformation for such pair $(\o{L}, \o{M})$ are
\begin{equation} \label{sys:difficulty}
\left.
\begin{array}{ll}
  -q r_x + q^2 r_y + q_x r -b q_x+b_x q+q^2 ( b_y-aq_y-a_x)-q^3a_y+ & \\
                +q_yq_x-q_{xy}q &= 0 \ , \\
  -cq_x+(c-ar)q_yq +(ar+r_y)q_x+(c_y-ra_y)q^2+ & \\
                                              +(rr_y-ar_x-r_yb-r_{xy}-ra_x+c_x)q &= 0 \ .
\end{array}
\right\}
\end{equation}
\end{thrm}
\begin{proof} Compare the corresponding coefficients on the both sides of equality~(\ref{eq:main}).
\end{proof}

Darboux theorem~\ref{thm:main:Darboux} provides us with a particular solution of the system~(\ref{sys:difficulty}).
The following statement is Theorem~\ref{thm:main:Darboux} written out more explicitly for the case of $\o{M}$ of bi-degree $(1,1)$.
\begin{thrm}[Darboux main theorem for bi-degree $(1,1)$] \label{thm:Dar:11}
 Let $\o{L} \in K[D]$ be an arbitrary operator of the form~\eqref{op:L} and $\psi_1, \psi_2$
 be two linearly independent solutions of $\o{L} \psi = 0$.  Then there exists a Darboux transformation with
\[
\o{M} = D_x + \frac{\alpha}{d} \; D_y + \frac{\beta}{d} \ ,
\]
where
\[
\begin{array}{ll}
  d & = -\psi_1\psi_{2y} + \psi_2\psi_{1y} \ ,        \\
  \alpha &= \psi_1\psi_{2x} - \psi_2\psi_{1x} \ ,          \\
  \beta &= -\psi_{2x}\psi_{1y} + \psi_{2y}\psi_{1x} \ .   \\
\end{array}
\]
\end{thrm}
\begin{remark}
If we denote by $\psi$ the ratio of these particular solutions,
\[
\psi = \frac{\psi_2}{\psi_1} \ ,
\]
then $\o{M}$ in the statement of Theorem~\ref{thm:Dar:11} can be written in more simple form:
\[
 \o{M} = D_x - \frac{\psi_x}{\psi_y} D_y + \frac{\psi_{1y}\psi_x}{\psi_1 \psi_y}-\frac{\psi_{1x}}{\psi_1}\ .
\]
\end{remark}

\begin{remark}
\label{rem:difficulty}
In order to describe all Darboux transformation for $\o{M}$ of bi-degree $(1,1)$ we need to solve system~(\ref{sys:difficulty}) for $q, r$,
where $a, b, c$ are known and are not constants in general.
Usual differential elimination techniques does not lead to a general solution.

A different approach can be to notice that in the system~(\ref{sys:difficulty}) the second equation is 
non-linear in both $q, r$, while the first equation is nonlinear in $q$ only. The first equation is a linear first-order non-homogeneous PDE on $r$, 
and since we know its particular solutions
(Theorem~\ref{thm:Dar:11}), one may solve it in quadratures. These quadratures are expressed in terms of $q$,
and therefore, after substituting the expression for $r$ into the second equation one gets even more nonlinear, rather large, PDE.
\end{remark}

In the rest of the paper we shall be proving that the general solution of system~(\ref{sys:difficulty}) is given by
the class of particular solutions from Theorem~\ref{thm:Dar:11}.

\section{Gauge Transformations of Pairs and Corresponding Invariants}
\label{sec:gauge}

Our plan is to address our problem using invariants methods. In this section we study gauge transformations of pairs 
$(\o{L},\o{M})$, which are almost classical with the only difference that 
we apply them to the pairs of operators. These transformations are not strong enough to simplify our system significantly, 
and completely new transformations will be introduced in Sec.~\ref{sec:gauge:evolution}. 
However, we shall use gauge transformations of pairs too. 

\begin{definition} Given some operator $\o{R} \in K[D]$ and invertible function $g \in K$, the corresponding
\textit{gauge transformation} is defined as
\[
 \o{R} \to \o{R}^{g} \ , \; \o{R}^g = g^{-1} \circ R \circ g \ ,
\]
where $\circ$ denotes the operation of the composition of operators in $K[D]$. It is convenient to take $g$ in the
form $g=\exp(\alpha)$.
Then we shall avoid fractions while writing this transformation out on the coefficients of $\o{R}$.
 \end{definition}

Our first step towards simplification of the problem is the following simple observation.
\begin{lem}  Let $\o{L}  = D_{x} D_y + a D_x + b D_y + c \in K[D]$ and $\o{M}= D_x + q D_y + r \in K[D]$.
If a Darboux transformation exists for the pair $(\o{L},\o{M})$, then
one exists also for $(\o{M}^g,\o{L}^g)$, where $g$ is an arbitrary invertible element of $K$.
\end{lem}
\begin{proof} Indeed, from the Darboux equality~(\ref{eq:main}) for the pair~$(\o{M},\o{L})$, we have
\[
 g^{-1} \circ \o{N} \circ g \circ g^{-1} \circ \o{L} \circ g = g^{-1} \circ \o{L}_1 \circ g \circ g^{-1} \circ \o{M} \circ g \ ,
\]
and, therefore, $\o{N}^g \circ \o{L}^g = \o{L}_1^g \circ \o{M}^g$.
Recalling that gauge transformations do not change the symbol of an operator,
we conclude the proof of the lemma.
\end{proof}

Therefore, it is natural to consider our problem for the equivalence classes of the pairs~$(\o{M},\o{L})$.
In order to define every class uniquely we determine a generating set of all the invariants of these pairs under
the gauge transformations.
\begin{definition} Let $\o{R} \in K[D]$ be an operator and $T$ be some transformation acting on $K[D]$.
Then a function of the coefficients of $\o{R}$ and of the derivatives of these coefficients is called a
\textit{differential invariant}
if it is unaltered under the action of $T$ on $\o{R}$. \\
The sum, and the product of two differential invariants is an invariant, as well as a derivative of an invariant is also an invariant.
In the infinite set of all possible differential invariants there is some subset (not necessarily proper) of differential invariants which generate all
others using algebraic operations and derivatives. Such a subset we shall call a \textit{generating set of invariants}.
\end{definition}
\begin{thrm} Let $\o{L}  = D_{x} D_y + a D_x + b D_y + c \in K[D]$ and $\o{M}= D_x + q D_y + r \in K[D]$.
On the set of all pairs $(\o{L},\o{M})$ of such  operators consider the gauge transformation of those with function $\exp(\alpha)$ :
\[
\varphi(\alpha): (\o{M},\o{L}) \rightarrow (\o{M}^{\exp(\alpha)}, \o{L}^{\exp(\alpha)}) \ .
\]
The following functions are invariants and in addition form a generating set of all differential invariants for such transformations:
\begin{equation} \label{sys:inv}
\left.
\begin{aligned}
 q\ ,\\
 m  &= a_x -  b_y\ ,\\
 h  &= ab - c +  a_x\ ,\\
 R  &=  r-b-qa\ .\\
\end{aligned}
\right\}
\end{equation}
\end{thrm}
\begin{remark}
Functions $h= ab - c +  a_x$, $k= ab - c +  b_y$
are known as $h$- and $k$-Laplace invariants as they are invariants of operator $\o{L}$ considered individually (without $\o{M}$) under the gauge transformations.
Both of them are present here: $h$ is present in its original form, and $k$ is hidden in $m$ as $m=h-k$.
\end{remark}
\begin{proof} To find a generating set of differential invariants we use the method of regularized moving frames
introduced by Fels and Olver in~\cite{FO1}. A good overview of recent developments in the area can be found in~\cite{Mansf_book}.
Note that our case is infinite dimensional, so the connected difficulties have been treated
in~\cite{OP:05}.

The transformations in question can be defined coordinate-wise as follows.
\[
\begin{aligned}
  &a_1 =  a + \alpha_y ,\\
  &b_1 =  b + \alpha_x ,\\
  &c_1 =  c + a\alpha_x + b\alpha_y + \alpha_{xy} + \alpha_x \alpha_y\ ,\\
  &q_1 =  q \ , \\
  &r_1 =  r+\alpha_x+q \alpha_y \ , \\
\end{aligned}
\]
where $\o{L}^{\exp(\alpha)}=D_x D_y +a_1D_x+ b_1D_y+ c_1$ and $\o{M}^{\exp(\alpha)}=D_x + q_1 D_y +r_1$.
We are choosing a cross-section as follows
\begin{equation} \label{eq:frame1}
\begin{aligned}
 &(a_1)_J = 0 \ ,  \\
 &(b_1)_X = 0 \ ,
\end{aligned}
\end{equation}
where $J$ is a string of the form $\underbrace{x \dots x}_{n} \underbrace{y \dots y}_{m}$, where $n=0,1,2, \dots$,
and where $m=0,1,2, \dots$ and
$X$ is a string of the form $\underbrace{x \dots x}_{l}$, where $l=0,1,2, \dots$.
For every $f \in K[D]$, the notation $f_J$ stands for the mixed derivative of $f$: of order $n$ order with respect to $x$ and
 $m$ with respect to $y$, while
$f_X$ stands for the $l$th derivative of $f$ with respect to $x$.

This gives us non-contradictory all the values for the parameters of the pseudo-group action, $\alpha_x, \alpha_y, \alpha_{xx}, \dots$:
\[
 \begin{aligned}
 &\alpha_x    = -b \ , \\
 &\alpha_y    = - a \ , \\
 &\alpha_{xy} = -a_x \ , \\
 &\dots
 \end{aligned}
\]
while the value for $\alpha$ we choose arbitrary, as it does not appear explicitly in the definition of the pseudo-group action.
Then we evaluate the edge invariants on the frame:
\[
\begin{aligned}
&(b_1)_y =b_y + \alpha_{xy} = b_y - a_x \ , \\
& c_1    =c - ab - ab -a_x +ab = c -a_x -ab \ , \\
& r_1    =r-b-qa \ ,
\end{aligned}
\]
which constitute the generating set of differential invariants of the pair under the gauge-transformations of the pair.
Invariants $m$ and $h$ differ by a sign from the first two we have just obtained, the third invariant we have obtained is exactly $R$ from the statement of
the theorem.
\end{proof}
\begin{definition} We shall call invariants~(\ref{sys:inv}) the \textit{gauge invariants of the pair}.
\end{definition}

Express the coefficients of the pair $(\o{L},\o{M})$ in terms of invariants:
\[
r = b + qa + R \ , \; c = ab - h + a_x \ .
\]
After these substitutions system~(\ref{sys:difficulty}) does not depend on $b$ itself, but only on its derivatives $b_y$ and $b_{xy}$.
Therefore, we can effectively use gauge-invariant $m$ by enforcing substitution
\[
b_y = a_x-m \ .
\]
Then system~\eqref{sys:difficulty} simplifies to the following one.
\[
\left.
\begin{aligned}
   \Omega &= 0\ , \\
 \Omega a &+ q_xh-qh_x-q^2h_y-q_xm+q_xR_y+\\
          & +qm_x-qR_{xy}-q_yqh -qRm+qRR_y = 0\ ,
\end{aligned}
\right\}
\]
where $\Omega=-2q^2m+q^2R_y+q_xR+q_yq_x-qR_x-q_{xy}q$.
Therefore, this system can be simplified further:
\begin{equation} \label{sys:difficulty2}
\left.
\begin{aligned}
& -2q^2m+q^2R_y+q_xR+q_yq_x-qR_x-q_{xy}q = 0\ , \\
& q_xh-qh_x-q^2h_y-q_xm+q_xR_y+ \\
& + qm_x-qR_{xy}-q_yqh -qRm+qRR_y = 0\ ,
\end{aligned}
\right\}
\end{equation}

In our problem  Darboux transformation, operator $\o{L}$ is considered to be given, therefore, the gauge-invariants $h$ and $m$ are given, and
the problem is reduced to the search of the general solution for system~(\ref{sys:difficulty2})
with respect to unknowns $q$ and $R$.

Although system~(\ref{sys:difficulty2}) is visually shorter than system~(\ref{sys:difficulty}),
it is still hard to solve using the usual methods, such as
the differentiation-cancellation technique.

Note that invariantiation and in particular moving frames method have been useful for investigation of Darboux like methods
earlier, see for example \cite{olver2011diff_inv_algebras}, \cite{2011:vassiliou}, \cite{kamran2011}. 

\section{Gauged Evolution of Pairs and Corresponding Invariants}
\label{sec:gauge:evolution}
\begin{lem}
\label{lem:gauged_evolution:DT_inv}
Let $\o{L}, \o{M} \in K[D]$ be two arbitrary operators in $K[D]$ and let $\o{L}$ be of order larger than one, while
$\o{M}$ is a first-order operator. Then if a Darboux transformation exists for the pair $(\o{L},\o{M})$, it also exists for
the pair $(\o{L}+\beta \o{M}, \o{M})$, where $\beta \in K$ is arbitrary.
\end{lem}
\begin{proof}
The existence of a Darboux transformation for the pair $(\o{L},\o{M})$ means that for some $\o{N}, \o{L}_1 \in K[D]$, where
$\o{L}_1$ has the same symbol as $\o{L}$. Therefore,
\[
\o{N} \circ \o{L} = \o{L}_1 \circ \o{M} \ .
\]
Then
\begin{equation}\label{eq:lemma:beta:eq}
\o{N} \circ (\o{L} + \beta \circ \o{M})  = \o{L}_1 \circ \o{M} + \o{N} \circ \beta \circ \o{M} = (\o{L}_1 + \o{N} \circ \beta) \circ \o{M} \ .
\end{equation}
Since $\o{N}$ must be of the same order as $\o{M}$, $\o{N}$ is a first-order operator. In addition, $\beta$ is zero-order operator. Therefore,
the symbol of the operator $\o{L}_1+ \o{N} \circ \beta$ is the same as the symbol of $\o{L}_1$, which is the same as the symbol of $\o{L}$. Therefore,
equality~(\ref{eq:lemma:beta:eq}) defines a Darboux transformation for pair $(\o{L}+\beta \o{M}, \o{M})$.
\end{proof}

\begin{definition} These transformation on the pairs, that is $(\o{L},\o{M}) \to (\widetilde{\o{L}},\widetilde{\o{M}})$:
\begin{eqnarray*}
 \widetilde{\o{L}} &\to& \o{L} + \beta \circ \o{M} \ ,\\
 \widetilde{\o{M}} &\to& \o{M} \ .
\end{eqnarray*}
we shall call \textit{evolution of the pair} (or \textit{$\beta$-evolution of the pair}).
\end{definition}
\begin{definition} Let $\o{L}  = D_{x} D_y + a D_x + b D_y + c \in K[D]$ and $\o{M} \in K[D]$ is arbitrary.
On the set of all the pairs $(\o{L},\o{M})$ of such operators consider
the consequential application of the gauge transformations and of the evolution:
for given $\alpha, \beta \in K$:
\begin{equation} \label{eq:evolution}
(\o{L},\o{M}) \mapsto (\o{L}^{\exp(\alpha)}+ \beta \o{M}^{\exp(\alpha)},\o{M}^{\exp(\alpha)}) \ .
\end{equation}
We shall call these transformations \textit{gauged evolution of the pairs}.
\end{definition}
\begin{thrm} 
 Let $\o{L}  = D_{x} D_y + a D_x + b D_y + c \in K[D]$ and $\o{M}= D_x + q D_y + r \in K[D]$. 
 The gauged evolutions of the pairs $(\o{L},\o{M})$ of such operators have the following generating set of differential invariants:
\begin{equation} \label{sys:inv:evolution}
\left.
\begin{aligned}
 I_1 &= q \ , \\
 I_2 &= 2m-R_y + \left(\frac{R}{q}\right)_x \ ,\\
 I_3 &= 2h+\left(\frac{R}{q}\right)_x-\frac{R^2}{2q} \ ,\\
\end{aligned}
\right\}
\end{equation}
\end{thrm}
\begin{remark} Notice that gauged evolution generating invariants~\eqref{sys:inv:evolution} are expressed in terms of generating 
gauge invariants $q,h,m,R$ only. This means that the gauged evolutions split the set of pairs $(\o{L},\o{M})$ into larger equivalence  classes than
the gauge transformations of pairs do. Also we see that those ``small'' gauge classes can belong to the ``larger'' gauged evolution classes only entirely. 
\end{remark}
\begin{proof} Evolution~(\ref{eq:evolution}) can be defined coordinate-wise as follows:
\[
\begin{aligned}
  &a_1 =  a+ \alpha_y+\beta \ ,\\
  &b_1 =  b+\alpha_x +\beta q \ ,\\
  &c_1 =  c+a \alpha_x+b\alpha_y+\alpha_{xy}+\alpha_y \alpha_x+\beta r+\beta \alpha_x + \beta q\alpha_y \ ,\\
  &q_1 =  q \ , \\
  &r_1 =  r+\alpha_x+q \alpha_y\ , \\
\end{aligned}
\]
where $\o{L}^{\exp(\alpha)}+ \beta \o{M}^{\exp(\alpha)}=D_x D_y +a_1D_x+ b_1D_y+ c_1$ and $\o{M}^{exp(\alpha)}=D_x + q_1 D_y +r_1$.
We are setting a cross-section by setting most of the coordinate functions to zero:
\[
\left.
\begin{aligned}
  &(a_1)_J =  0 \ ,\\
  &(b_1)_J =  0 \  ,\\
  &(r_1)_X =  0 \ ,
\end{aligned}
\right\}
\]
where $J$ and $X$ are the same notations as in~(\ref{eq:frame1}).
Then at the beginning we have three equations,
\[
 a_1 =0 \ , \; b_1=0 \ , \; r_1=0
\]
and three variables, parameters to determine:
\[
 \beta \ , \; \alpha_x \ , \; \alpha_y \ .
\]
The determinant is not $0$, so there is a unique solution
for such a system.
At the next step we consider first prolongations only, which gives us
$5$ equations for $5$ variables and this linear system has non-zero determinant.
In general, considering $i$-th prolongation we have $2i+3$ variables and the same number of equations, and a non-zero
determinant of the corresponding linear system.
Therefore, we have defined a frame, and the generating set of invariants in this case consists of the corner invariants:
\[
\left.
\begin{aligned}
 I^q &= q \ , \\
 I^r_y &= r_y-a_x-\frac{r_x+b_x}{q}+\frac{q_xr-q_xb}{q^2}-q_ya-qa_y+b_y\ ,\\
 I^c&=c-\frac{a_x}{2}-\frac{br}{2q}-\frac{q_x b}{q^2}+\frac{q_x r}{2 q^2}-\frac{ab+ar}{2}+\frac{qa^2}{4}+\frac{b^2}{4q}-\frac{r_x}{2q}
+\frac{b_x}{2q} + \frac{r^2}{4q} \ .\\
\end{aligned}
\right\}
\]
Substituting
\[
\begin{aligned}
 & r=b+qa+R \ , \\
 & c=ab-h+a_x \ , \\
 & a_x = b_y +m \ .
\end{aligned}
\]
we obtain (up to a sign and a multiplication by $2$) the invariants claimed in the statement of the theorem.
\end{proof}
\begin{definition} We shall refer to invariants~(\ref{sys:inv:evolution}) as \textit{gauged evolution invariants}.
\end{definition}
\section{Solution of the PDE System. Description of All Darboux transformations of Total Order Two.}

In Lemma~\ref{lem:gauged_evolution:DT_inv} we showed that the property of the existence of a Darboux transformation for a pair is 
invariant under the gauged evolutions. This does not necessarily mean that there is some explicit invariant form for system~\eqref{sys:difficulty}.
Theorem below demonstrates, however, that in this particular case, we can have such explicit invariant form.
We also see that the invariantizing system can be written in much simpler form than system~\eqref{sys:difficulty}.
\begin{thrm}[Necessary and sufficient conditions for the existence of a Darboux transformation in terms of evolution]
\label{thm:last:conds}
Given pair $(\o{L},\o{M})$,
where $\o{L}  = D_{x} D_y + a D_x + b D_y + c \in K[D]$ and $\o{M}= D_x + q D_y + r \in K[D]$, 
there exists a corresponding Darboux transformation if and only if its evolution invariants $(q,I_2,I_3)$ satisfy the following two
conditions simultaneously:
\begin{eqnarray}
 &&I_2 + Q_{xy}= 0\ , \label{eq:i2} \\
 &&I_{3,x} + q I_{3,y} + (q_y-q_x/q) I_3 = Q_x Q_{xy} -Q_{xxy} \ , \label{eq:i3}
\end{eqnarray}
where $Q=\ln q$.
\end{thrm}
\begin{proof} Expressing $m$ and $h$ using the second and the third equations of~(\ref{sys:inv:evolution}) and using
$b_y = a_x-m$ system~(\ref{sys:difficulty2}) can be written as in the statement.
\end{proof}

We also invariantize the class of particular solutions for system~(\ref{sys:difficulty}) which we derived from Darboux Wronskian formulas:
\begin{thrm}[Darboux transformations constructed from Wronskians]
\label{thm:DT:from:W}
Let $\o{L} = D_{x} D_y + a D_x + b D_y + c \in K[D]$ and let $\psi_1, \psi_2$ be two linearly independent elements of
its kernel.
Let ${\displaystyle \psi=\frac{\psi_2}{\psi_1}}$ and ${\displaystyle A= \frac{ \psi_{xy}}{\psi_x}}$ and
${\displaystyle B=\frac{ \psi_{xy}}{\psi_y}}$. Then for $\o{L}$ there exists a Darboux transformation such that the evolution invariants of the corresponding pair
$(\o{M},\o{L})$ are as follows:
\begin{equation} \label{sys:DT:Wronskian}
\left.
\begin{aligned}
 q   &= -  \frac{B}{A} \ , \\
 I_2 &=  B_y - A_x \ , \\
 I_3 &=  - A_x +\frac{A \cdot B}2 \ . \\
\end{aligned}
\right\}
\end{equation}
\end{thrm}
\begin{proof} Compute the values of the gauged evolution invariants~(\ref{sys:inv:evolution}) for $\o{M}$
constructed using Darboux formulas given in Theorem~\ref{thm:Dar:11}. Then we have
\begin{align*}
&q= \frac{-\psi_x}{\psi_y}\ ,\\
&I_2=  \frac{-\psi_{xxy}}{\psi_x}
                      +\frac{ \psi_{xx}\psi_{xy}}{\psi_x^2}
                      +\frac{ \psi_{xyy}}{\psi_y}
                      -\frac{ \psi_{xy}\psi_{yy}}{\psi_y^2}\ ,\\
&I_3 = \frac{-\psi_{xxy}}{\psi_x}
                      +\frac{ \psi_{xx}\psi_{xy}}{\psi_x^2}
                      +\frac{ \psi_{xy}^2}{2\psi_x \psi_y}\ ,
\end{align*}
 which can be re-written in a very short form using notations $A$ and $B$.
\end{proof} 

Value of $I_3$ from Theorem~\ref{thm:DT:from:W} is a particular solution of~(\ref{eq:i3}), a first-order linear non-homogeneous PDE on $I_3$. 

Let us solve~(\ref{eq:i3}). One useful idea is to consider $q$ in the form ${ \displaystyle q   = -\frac{z_x}{z_y} }$, where $z$ is not required to be a ratio
of two particular solutions of $\o{L}(\psi)=0$. For this $q \in K$, invariants $I_1, I_2$ can be computed straightforwardly:
$I_1=q, I_2=-(\ln q)_{xy}$. 

Equation~(\ref{eq:i3}) is a first order non-homogeneous, and the general solution can be obtained as the sum
of its particular solution
and of the general solution of the corresponding homogeneous PDE.
As a particular solution we take the expressions from~(\ref{sys:DT:Wronskian}). Note that for this particular solution
$z=\psi$, where $\psi$ is some particular solution of $\o{L} \psi = 0$.

Using expression for $q$ in terms of $z$ the homogeneous PDE corresponding to~(\ref{eq:i3}) can be written as
$ { \displaystyle I_3 \left( \frac{z_{xx}}{z_x} - \frac{z_x z_{yy}}{z_y^2} \right) - I_{3,x} + \frac{z_x}{z_y} I_{3,y} =0}$.
Considering  $I_3$ in the form $I_3 = e^J$ for suitable $J \in K[D]$ and assuming
$z_y \neq 0$, the PDE is equivalent to
\begin{equation} \label{eq:1}
 T -  \frac{J_x}{z_x} + \frac{J_y}{z_y}  = 0 \ ,
\end{equation}
the non-homogeneous part of the equation being $ { \displaystyle T = \frac{z_{xx}}{z_x^2} -  \frac{z_{yy}}{z_y^2} }$.
After applying the method of characteristics, which we follow more carefully in the proof of the next theorem,
 Theorem~\ref{thm:simple},
we choose change of variables $\xi = x, \eta= z(x,y)$, and, in the new variables, equation~(\ref{eq:1}) has the form
\[
\frac{J_\xi}{ z_x(\xi,\eta)} = T(\xi,\eta) \ ,
\]
and, therefore,
\[
J = \int T(\xi,\widetilde{y}(\xi,\eta))   z_\xi(\xi,\widetilde{y}(\xi,\eta)) d \xi + F(\eta) \ ,
\]
where $\widetilde{y}=\widetilde{y}(\xi,\eta)$ is the solution of $z(\xi,y)-\eta$ for $y$ and $F(\eta)$ is an arbitrary function of $\eta$. Changing variables back we have
\[
J = \int^x T(\xi,\widetilde{y}(\xi,z))   z_\xi(\xi,\widetilde{y}(\xi,z)) d \xi + F(z) \ .
\]
This means that the homogeneous part of PDE~(\ref{eq:i3}) has the general solution
\[
 I_3 = G(z) \cdot \exp \left(\int^x T(\xi,\widetilde{y}(\xi,z))   z_\xi(\xi,\widetilde{y}(\xi,z)) d \xi \right) \ ,
\]
where $G(z)$ is an arbitrary function of $z$.
Then since we know the particular solution, $I_{30}(\psi)$, where $\psi =\psi_2/\psi_1$ for two particular solutions of $\o{L} u =0$, the general description of
all Darboux transformations of total order two will be as follows:
\begin{equation}
 \label{sys:straightf}
\left.
\begin{aligned}
 q   &= {\displaystyle -\frac{z_x}{z_y}}, \\
 I_2 &= {\displaystyle \left( \frac{ z_{xy}}{z_y} \right)_y - \left( \frac{ z_{xy}}{z_x} \right)_x }, \\
 I_3 &= G(z) \cdot \exp \left(\int^x T(\xi,\widetilde{y}(\xi,z))   z_\xi(\xi,\widetilde{y}(\xi,z)) d \xi \right)+ I_{30}(\psi_2/\psi_1) \ .
\end{aligned}
\right\}
\end{equation}
However, formulae~\eqref{sys:straightf} does not serve either of our two purposes: short description of all possible Darboux transformations and
proof of completeness. Indeed, we would have to decide whether the series of particular solutions~\eqref{sys:DT:Wronskian}
is the same as the obtained general solution. 

Thus, we approach the solution of~\eqref{eq:i3}
in a different manner. The next theorem implies that we can construct a series of particular solutions of~\eqref{eq:i3}
using arbitrary function $z$ rather than $\psi$, which must be the ratio of two linearly independent particular solutions $\psi_1$ and $\psi_2$ of the initial PDE $\o{L}=0$.

\begin{thrm} \textbf{Another large class of particular solutions for the PDE on $I_3$.}
\label{thm:i30}
Let $z \in K$ be arbitrary and non-constant and $F=F(z)$ be an arbitrary function of $z$, then
\[
 q = -\frac{z_x}{z_y} = - \frac{(F(z))_x}{(F(z))_y}
\]
and function
\begin{equation}
\label{eq:i30}
  I_{30}(z) = - \frac{z_{xxy}}{z_x} +\frac{ z_{xx}z_{xy}}{z_x^2} +\frac{ z_{xy}^2}{2z_x z_y}
\end{equation}
gives particular solutions of~\eqref{eq:i3}.
More strongly, if instead of $z$ we use the argument $F(z)$, that is $I_{30}(z)$ is replaced by
\[
I_{30}(F(z)) \ ,
\]
we still have particular solutions of~(\ref{eq:i3}).
\end{thrm}
\begin{proof} Can be verified by the substitution.
\end{proof}

In other words, on the invariant level we can forget about the fact that $\psi$ must be a ratio of two solutions. Of  course such trick would not work 
in general for the pre-invariantized system:  if we take Darboux Wronskian Formulas and substitute arbitrary functions instead of solutions
we would not necessarily get a Darboux transformation. So the invariantization using gauged evolutions factors out some meaningful conditions, 
which justifies giving them a separate name. 

Using new class of particular solutions discovered in Theorem~\ref{thm:i30}, we can find the general solution for the invariantized system (see it in 
Theorem~\ref{thm:last:conds}) in the following form:

\begin{thrm} \textbf{Simple Description of all Darboux transformations of bi-degree $(1,1)$.}
\label{thm:simple}
All Darboux transformations of some $\o{L}  = D_{x} D_y + a D_x + b D_y + c \in K[D]$
generated by some $\o{M}$, $\o{M}= D_x + q D_y + r \in K[D]$ are parametrized by $z \in K$ an arbitrary non constant and can be written
as
\[
\begin{aligned}
 q   &= - {\displaystyle \frac{B}{A}} \ , \\
 I_2 &= {\displaystyle B_y - A_x }\ , \\
 I_3 &= {\displaystyle - A_x +\frac{A \cdot B}{2} } \ , \\
\end{aligned}
\]
where ${\displaystyle A= \frac{ z_{xy}}{z_x}}$ and ${\displaystyle B=\frac{ z_{xy}}{z_y}}$.
\end{thrm}
\begin{proof} Find $z \in K$ such that ${\displaystyle q = -\frac{z_x}{z_y}}$. Then
according to~(\ref{eq:i2}) we must have $I_2=-(\ln(q))_{xy}$, which in terms of $z$ has the form given for $I_2$ in
the statement of the theorem.

Now we solve~(\ref{eq:i3}) for $I_3$. The solution of this equation can be obtained as the sum of particular solution~(\ref{eq:i30})
and of the general solution of the corresponding homogeneous PDE,
\begin{equation} \label{eq:homo1}
 I_3 \left( \frac{z_{xx}}{z_x} - \frac{z_x z_{yy}}{z_y^2} \right) - I_{3,x} + \frac{z_x}{z_y} I_{3,y} =0 \ .
\end{equation}
Considering $I_3$ in the form $I_3 = e^J$ for suitable $J \in K[D]$ this equation can be re-written equivalently as
\begin{equation} \label{eq:1bis}
\frac{z_{xx}}{z_x^2} -  \frac{z_{yy}}{z_y^2} -  \frac{J_x}{z_x} + \frac{J_y}{z_y}  = 0 \ .
\end{equation}

Again we have homogeneous and non-homogeneous parts, and therefore, the general solution can be represented as
the sum of a particular solution of the non-homogeneous part and the general solution of the homogeneous one.

Using the methods of characteristics one can find the general solution of the homogeneous part of~(\ref{eq:1}).
Consider the equality
\[
\frac{dx}{-1/z_x} = \frac{dy}{1/z_y} \ ,
\]
which can be rewritten in the form $z_x dx + z_y d y  = d(z) = 0$, and, therefore, $z = z(x,y) = C$, where $C$ is a constant.
Therefore, we consider the following change of variables:
\[
 \left.
\begin{aligned}
 \xi &= x \ , \\
 \eta  &= z(x,y) \ ,
\end{aligned}
 \right\}
\]
which is non-degenerate, since the Jacobian is nonzero:
\[
\frac{\partial(\xi,\eta)}{\partial(x,y)}  = \left|
\begin{aligned}
  \xi_x \; & \; \xi_y \\
  \eta_x \;  & \; \eta_y
\end{aligned}
\right| = \eta_y = z_y \neq 0  \ .
\]
In the new variables equation~(\ref{eq:1}) has the form
\[
\frac{J_\xi}{ z_x(\xi,\eta)} =0 \ ,
\]
and, therefore, the general solution is
\[
 J = H(z) \ ,
\]
where $H(z)$ is an arbitrary function of $z$.

Notice now that since both $I_{30}(z)$ and $I_{30}(F(z))$ are solutions of~(\ref{eq:i3}), therefore,
their difference is a solution of~(\ref{eq:homo1}): subtracting~(\ref{eq:i30}) from $I_{30}(G(z))$ we have
\begin{equation} \label{eq:i30d}
I_{30d} = \frac{1}{2} z_x z_y \left( \frac{2 F' F''' - 3 (F'')^2 }{(F')^2} \right) =   z_x z_y G(z) \ ,
\end{equation}
where $G(z)$ is an arbitrary function. Therefore,
\[
J_{30d} = \ln(I_{30d}) = G_1(z) + \ln(z_x z_y) \ ,
\]
where $G_1(z)$ is an arbitrary function, is a solution of~(\ref{eq:1}). Therefore, the general solution of~(\ref{eq:1})
is
\[
 J = G_1(z) + \ln(z_x z_y) + H(z) \ .
\]
Correspondingly, the general solution of~(\ref{eq:homo1}), which is the homogeneous part of the PDE
we needed to solve,~(\ref{eq:i3}) is $I=\exp(J)$, which is
\[
 I = G(z) z_x z_y \ .
\]
Adding to this the particular solution of the non-homogeneous part, $I_{30}$, we conclude that
\[
I_3 = z_x z_y G(z) + I_{30}(z) = z_x z_y \frac{2 F' F''' - 3 (F'')^2 }{2(F')^2}  + I_{30}(z) \ .
\]
Notice that this expression is exactly $I_{30}(F(z))$. Now we proved that
\[
\begin{aligned}
 q   &= {\displaystyle -\frac{z_x}{z_y}}, \\
 I_2 &= {\displaystyle \left( \frac{ z_{xy}}{z_y} \right)_y - \left( \frac{ z_{xy}}{z_x} \right)_x }, \\
 I_3 &= I_{30}(F(z)) \ .
\end{aligned}
\]
However, notice that the application of $q$ to $z$ and $q$ to $F(z)$ gives the same result, that is
$q(F(z))=q(z)$, the same with $I_2$: $I_2(F(z))=I_2(z)$. Therefore, we can
substitute $\widetilde{z}=F(z)$ and have the statement of the theorem, which we write in terms of
$z$ again for convenience.
\end{proof}

Comparing Theorem~\ref{thm:simple} with Theorem~\ref{thm:DT:from:W},
we see that the general solution of the invariantized system  (see it in 
Theorem~\ref{thm:last:conds}) is much richer than the class of the invariantized particular solutions~\eqref{sys:DT:Wronskian}.
Does this mean that our hypothesis was wrong and that there must be something else besides Darboux transformations generated by
Darboux Wronskians formulas? In the following theorem lifting our results back into pre-invariantized situation we conclude the proof 
of our hypothesis.

\begin{thrm}
\label{thm:completeness}
Given operator $\o{L}  = D_{x} D_y + a D_x + b D_y + c \in K[D]$. Every its Darboux transformation
generated by $\o{M}$, $\o{M}= D_x + q D_y + r \in K[D]$ is described by Darboux Wronskian formulas.
\end{thrm}
\begin{proof}
Let $\o{N} \o{L} = \o{L}_1 \o{M}$ be some Darboux transformations with the gauged evolution invariants $(q,I_2,I_3)$.
According to Theorem~\ref{thm:simple} there exists $z \in K$ such that
\[
\begin{aligned}
 q   &=  -\frac{z_x}{z_y} =  -\frac{B}{A} \ , \\
 I_2 &=  \left( \frac{ z_{xy}}{z_y} \right)_y - \left( \frac{ z_{xy}}{z_x} \right)_x  = B_y - A_x \ , \\
 I_3 &= I_{30}(z) =  -A_x + \frac{A\cdot B}{2} \ ,
\end{aligned}
\]
where ${\displaystyle A = \frac{z_{xy}}{z_x}}$, ${\displaystyle B = \frac{z_{xy}}{z_y}}$.
Comparing this with the invariants given in Theorem~\ref{thm:DT:from:W}, we conclude
that there are some Darboux transformation constructed by Darboux formulae that has the same gauged evolution invariants.

Now let us find among the pairs that are constructed using Darboux formulas
and having these gauged evolution invariants $(q,I_2,I_3)$ those which has the
same gauge invariants as our initial pair $(\o{L},\o{M})$.

Let us denote the gauge invariants of the pair $(\o{L},\o{M})$ by $(m_0,h_0,q_0,R_0)$. Then
$I_2$ and $I_3$ can be expressed in terms of $(m_0,h_0,q_0,R_0)$, see~(\ref{sys:inv:evolution}).
Therefore,
\[
\begin{array}{ll}
 2m_0  &= I_2 - (R_0/q)_x + R_{0y} \ ,\\
 2h_0  &= I_3 + R_0^2/(2q) - (R_0/q)_x\ .\\
\end{array}
\]
Now, since $(I_2, I_3)$ are given in terms of $z$, then invariants $(m_0,h_0)$
are given in terms of $z$ and $R_0$.

That is it is enough to find among the pairs that are constructed using Darboux Wronskians formulas a pair with the same $R_0$.

Choose arbitrary functions $z_1$ and $c_0$ and construct an operator $\o{L}'$
of the form~(\ref{op:L}) which has two solutions: $z_1$ and $z z_1$ and that coefficient $c=c_0$.
Using the same pair of solutions construct $\o{M}'$ using Darboux formulas, see Theorem~\ref{thm:Dar:11}.

Then $R'$ corresponding to the pair $(\o{L}',\o{M}')$ can be expressed in terms of $z$, $z_1$, and $c_0$
in the form
\[
R' =- \frac{2 z_1 z_x}{-z_x z_{1,y}+z_y z_{1,x}} \cdot c_0 + T(z,z_1) \ ,
\]
where $T(z,z_1)$ certain expression depending on $z$ and $z_1$ only.
This means that for every $z_1$ we can uniquely find such $c$ that $R'=R_0$.

Therefore, for arbitrary $z_1$ there exist a Darboux transformation $(\o{L}',\o{M}')$ constructed using Darboux formulas for which
all gauge invariants $(h,k,q,R)$ of the pair are correspondingly the same as those of the initial pair $(\o{L},\o{M})$.

Since those agree, then $(\o{L},\o{M})$ is different from $(\o{L}',\o{M}')$ by a gauge transformation, and therefore,
 $(\o{L},\o{M})$ can be also constructed using Darboux Wronskian formulas.
\end{proof}

\section{Conclusions}

The present paper closes an essential question for the theory of Darboux transformations: Darboux Wronskians formulas are complete for 
Darboux transformation of total order two of operators $\o{L} = D_{x} D_y + a D_x + b D_y + c$ with non-constant coefficients.
Since for Darboux transformation of total order there are two famous exceptions: Laplace transformations, the case of 
the total order two has been crucial. 

We saw that newly introduced transformations of pairs, gauged evolutions may have much deeper role than just a tool 
in proof of our specific problem (Theorem~\ref{thm:i30} and the paragraph after it).

We found a very short invariant description of all possible Darboux transformation for $\o{L} = D_{x} D_y + a D_x + b D_y + c$
generated by $\o{M}$ in the form $\o{M}= D_x + q D_y + r \in K[D]$ (Theorem~\ref{thm:simple}).

Now it is natural to expect completeness of Darboux Wronskians formulas for transformations
of orders higher than two. We expect this one to be rather difficult to prove. Simple repetition and adjustments of the methods and 
ideas of this work would not work. For example, one of the crutial points was the introduction of the gauged evolutions, which cannot be defined for pairs $(\o{L},\o{M})$ if $\o{M}$ has order larger than $\o{L}$. 

\bibliographystyle{natbib}  
\bibliography{/home/ekaterina/Desktop/work/general.bib}{}

\end{document}